\documentclass[a4paper]{amsart}

\usepackage{latexsym, amssymb,amsmath, amsthm,amsfonts}
\usepackage{tikz}
\tikzset{My Line Style/.style={smooth}}
\usetikzlibrary{shapes,arrows}

\newcommand{\kk}{\Bbbk}

\def\SL{\operatorname{SL}}

\def\SL2{\operatorname{SL}_{2}(\kk)}

\def\GL2{\operatorname{GL}_2}

\def\INVSL2{$\kk[V]^{operatorname{SL}_{2}(\kk)}$}
\def\INVSO2{$\kk[V]^{operatorname{SO}_{2}(\kk)}$}
\def\INVGL2{$\kk[V]^{operatorname{GL}_{2}(\kk)}$}

\def\Hom{\operatorname{Hom}}
\def\GL{\operatorname{GL}}
\def\SL{\operatorname{SL}}
\def\id{\operatorname{id}}
\def\codom{\operatorname{codom}}

\def\Id{\operatorname{Id}}

\def\res{\operatorname{res}}
\def\Tr{\operatorname{Tr}}
\def\stmod{\operatorname{stmod}}

\def\coker{\operatorname{coker}}

\newtheorem{Lemma}{Lemma}[section]
\newtheorem{Theorem}[Lemma]{Theorem}
\newtheorem{Corollary}[Lemma]{Corollary}

\newtheorem{Prop}[Lemma]{Proposition}
\newtheorem{Conj}[Lemma]{Conjecture}

\theoremstyle{definition}

\theoremstyle{remark}
  \newtheorem{rem}[Lemma]{Remark}

\newtheoremstyle{Acknowledgments}% name
  {}% {\topsep}%      Space above
    {}% {\topsep}%      Space below
     {}%         Body font
     {}%         Indent amount (empty = no indent, \parindent = para indent)
    {\bfseries}% Thm head font
    {}%        Punctuation after thm head
     {.5em}%     Space after thm head: " " = normal interword space;
     {\thmname{#1}\thmnumber{ }\thmnote{ (#3)}}%         Thm head spec (can be left empty, meaning `normal')
\theoremstyle{Acknowledgments}
\newtheorem{ack}{Acknowledgments.}

\title[Symmetric powers for elementary abelian $p$-groups]
{Symmetric powers and modular invariants of elementary abelian $p$-groups}

\author{Jonathan Elmer}
\address{Middlesex University\\
The Burroughs, London\\
NW4 4BT}
\email{j.elmer@mdx.ac.uk}

\date{\today}
\subjclass[2010]{20C20, 13A50}
\keywords{modular representation theory, invariant theory, elementary abelian $p$-groups, symmetric powers, relative stable module category}

\begin{document}
\maketitle

\begin{abstract} Let $E$ be a elementary abelian $p$-group of order $q=p^n$. Let $W$ be a faithful indecomposable representation of $E$  with dimension 2 over a field $\kk$ of characteristic $p$, and let $V= S^m(W)$ with $m<q$. We prove that the rings of invariants $\kk[V]^E$ are generated by elements of degree $\leq q$ and relative transfers. This extends recent work of Wehlau \cite{WehlauCyclicViaClassical} on modular invariants of cyclic groups of order $p$. If $m<p$ we prove that $\kk[V]^E$ is generated by invariants of degree $\leq 2q-3$, extending a result of Fleischmann, Sezer, Shank and Woodcock \cite{FleischmannSezerShankWoodcock} for cyclic groups of order $p$ . Our methods are primarily representation-theoretic, and along the way we prove that  for any $d<q$ with $d+m \geq q$, $S^d(V^*)$ is projective relative to the set of subgroups of $E$ with order $ \leq m$, and that the sequence $S^d(V^*)_{d \geq 0}$ is periodic with period $q$, modulo summands which are projective relative to the same set of subgroups. These results extend results of Almkvist and Fossum \cite{AlmkvistFossum} on cyclic groups of prime order. 
\end{abstract}

\section{Introduction}
\subsection{Modular invariant theory and degree bounds}

Let $G$ be a finite group and $\kk$ a field of arbitrary characteristic. Let $V$ be a finite-dimensional represenation of $G$, which in this article will always mean a left $\kk G$-module. We denote by $\kk[V]$ the $\kk$-algebra of polynomial functions $V \rightarrow \kk$. This itself becomes a $\kk G$-module with the action given by $$(\sigma f)(v) = f(\sigma^{-1}v)$$ for $f \in \kk[V], v  \in V$ and $\sigma \in G$. 

Let $\{x_1,x_2,\ldots,x_n\}$ be a basis for $V^*$. If $\kk$ is infinite, we can identify $\kk[V]$ with the polynomial ring $\kk[x_1,x_2,\ldots,x_n]$; this is graded by total degree, and the action of $G$ on $\kk[x_1,x_2,\ldots,x_n]$ is by graded algebra automorphisms. As a $\kk G$-module, the homogeneous component of degree $d$ in $\kk[V]$ is isomorphic to $S^d(V^*)$, the $d$th symmetric power of $V^*$.

The set of fixed points $\kk[V]^G$ forms a $\kk$-subalgebra of $\kk[V]$ called the \emph{algebra of invariants}. This is the central object of study in invariant theory. The most natural goal in invariant theory is to compute algebra generators of $\kk[V]^G$. This is a hard problem in general, especially if $|G|$ is divisible by $\kk$. For example, even for cyclic groups of order $p$ the list of modular representations $V$ for which algebra generators of $\kk[V]^G$ are known is rather short, see \cite{WehlauCyclicViaClassical}. Some general results are known, however. Famously, Noether proved that the ring of invariants $\mathbb{C}[V]^G$ has a generating set consisting of invariants of degree $\leq |G|$, for any representation of $G$ over $\mathbb{C}$. For this reason, the minimum $d$ such that $\oplus_{i=0}^d \kk[V]^G_i$ generates $\kk[V]^G$ as a $\kk$-algebra is called the \emph{Noether bound} and written as $\beta(\kk[V]^G)$. Her results were extended independently by Fleischmann \cite{FleischmannNoetherBound} and Fogarty \cite{FogartyNoetherBound} to any field whose characteristic does not divide the order of $G$, the so called non-modular case. In contrast, it is known that in the modular case no general bound on the degrees of generators which depends on $|G|$ alone exists. Recent work of Symonds \cite{SymondsRegularity} has shown that, independently of the characteristic of $\kk$, $\kk[V]^G$ is generated by invariants of degree at most $\max(|G|, (|G|-1)\dim(V))$.

Fleischmann, Sezer, Shank and Woodcock \cite{FleischmannSezerShankWoodcock} proved that if $V$ is an indecomposable modular representation of a cyclic group of order $p$, then $\kk[V]^G$ is generated by invariants of degree at most $2p-3$. In particular this shows that Symonds' bound is far from sharp. One goal of this article is to prove the following result:

\begin{Theorem}\label{thm:2q-3} Let $E$ be an elementary abelian $p$-group of order $q$, $\kk$ an infinite field of characteristic $p$, $W$ a faithful indecomposable $\kk E$-module with dimension 2. Let $V=S^m(W)$ with $m<p$. Then $\kk[V]^E$ is generated by invariants of degree $\leq 2q-3$. 
\end{Theorem}
Note that if $E = C_p$ is a cyclic group of order $p$ there is only one isomorphism class of faithful indecomposable representation of dimension 2. Furthermore, if $V$ is any indecomposable $\kk C_p$-module then $V=S^m(W)$ for some $m<p$. So the above generalises \cite[Proposition~1.1(a)]{FleischmannSezerShankWoodcock}.

\subsection{The transfer}\label{subsec:transfer}
Let $G$ be a finite group, $H \leq G$ and $M$ a $\kk G$-module. We denote the set of $G$-fixed points in $M$ by $M^G$. There is a $\kk G$-map $M^H \rightarrow M^G$ defined as follows:
\[Tr^G_H(f) = \sum_{\sigma \in S} \sigma f\] where $f \in M$ and $S$ is a left-transversal of $H$ in $G$. This is called the relative trace or transfer. It is clear that the map is independent of the choice of $S$. If $H=1$ we usually write this as $\Tr^G$ and call it simply the trace or transfer.

In case $M=\kk[V]$ this restricts to a degree-preserving $\kk[V]^G$-homomorphism $\kk[V]^H \rightarrow \kk[V]^G$, whose image is an ideal of $\kk[V]^G$. We denote its image by $I^G_H$ . More generally, given a set $\mathcal{X}$ of subgroups of $G$, we set $I^G_{\mathcal{X}} = \sum_{H \in \mathcal{X}} I^G_H$.

If $|G:H|$ is not divisible by char$(\kk)$ then $\Tr^G_H$ is surjective. This has many nice consequences; in particular, it implies that $\kk[V]^G$ is a direct summand of $\kk[V]^H$ as a $\kk[V]^G$-module, and hence that $\kk[V]^G$ is Cohen-Macaulay if $\kk[V]^H$ is. It also shows that in the non-modular case, every invariant lies in the image of the transfer map and every ring of invariants is Cohen-Macaulay.

From now on suppose that $G$ is divisible by $p$ = char$(\kk)>0$. Choose a Sylow-$p$-subgroup $P$ of $G$ and denote by $I^G_{<P}$ the sum of all $I^G_Q$ with $Q<P$. It is easily shown that $I^G_{<P}$ is independent of the choice of $P$. The ring $\kk[V]^G/I^G_{<P}$ has attracted some attention in recent years. The prevailing idea is that it behaves in many ways like a non-modular ring of invariants. For example, Totaro \cite{totaro} has shown that $\kk[V]^G/I^G_{<P}$ is a Cohen-Macaulay ring for any $G$ and $V$, generalising earlier work of Fleischmann \cite{FleischmannTransfer}, where $I^G_{<P}$ is replaced by its radical. In the same spirit is the following conjecture, reported by Wehlau \cite{QueensReport}.

\begin{Conj}\label{conj:queens} Let $G$ be a finite group and $V$ a $\kk G$-module. Then $\kk[V]^G/I^G_{<P}$ is generated by invariants of degree $\leq |G|$.
\end{Conj}
 
It had earlier been shown that this holds whenever $V$ is an indecomposable representation of a cyclic group of order $p$. In the present article we prove
\begin{Theorem}\label{thm:quotientofkvg} Let $E$ be an elementary abelian $p$-group of order $q$, $\kk$ an infinite field of characteristic $p$, $W$ a faithful indecomposable $\kk E$-module with dimension 2. Let $V=S^m(W)$ with $m<q$. Then $\kk[V]^E/I^E_{<E}$ is generated by invariants of degree $\leq q$. 
\end{Theorem}
In other words, this class of representations of elementary abelian $p$-groups satistfies Conjecture \ref{conj:queens}. Once more, as every indecomposable representation of a cyclic group can be written as a  symmetric power of the unique indecomposable representation of degree 2, this is a generalisation of the earlier result.

\subsection{Symmetric powers and relative projectivity}\label{subsec:sympowers}

 Let $G$ be a finite group and let $V$ and $W$ be  finite-dimensional representations of $G$ over a field $\kk$. Let $v_1,v_2,\ldots, v_m$ and $w_1,w_2,\ldots, w_n$ be bases of $V$ and $W$ over $\kk$. Then the tensor product $V \otimes W$ of $V$ and $W$ is the $\kk$-vector space spanned by elements of the form $v_i \otimes w_j$, where scalar multiplication satisfies $\lambda (v_i \otimes w_j) = (\lambda v_i) \otimes w_j = v_i \otimes (\lambda w_j)$ for all $\lambda \in \kk$.  There is a linear action of $G$ on the space defined by $g(v \otimes w) = gv \otimes w + v \otimes gw$. We can take the tensor product of $V$ with itself, and iterate the construction $d$ times to obtain, for any natural number $d$, a module $T^d(V) = V \otimes \cdots \otimes V$, called the \emph{$d$th tensor power} of $V$. Formally, the $d$th \emph{symmetric power} $S^d(V)$ of $V$ is defined to be the quotient of $T^d(V)$ by the subspace  generated by elements of the form $v_1 \otimes \cdots \otimes v_d - v_{\sigma(1)} \otimes \ldots \otimes v_{\sigma(d)}$ where $\sigma \in \Sigma_d$, the symmetric group on $\{1,2,\ldots, d\}$. We have $S^0(V) \cong \kk$ for any $V$, and we use the convention that $S^d(V)=0$ for negative values of $d$.

Symmetric powers of indecomposable representations are not indecomposable in general, and a central problem in representation theory is to try to understand their indecomposable summands. If $|G|$ is invertible in $\kk$ then this is largely a matter of character theory. The first authors to consider the problem in the modular case were Almkvist and Fossum \cite{AlmkvistFossum}. In this remarkable work, the authors give formulae for the indecomposable summands of any representation of the form $V \otimes W$, $S^d(V)$ or $\Lambda^d(V)$ (exterior power), where $V$ and $W$ are indecomposable representations of a cyclic group of order $p$ over a field $\kk$ of characteristic $p$. Some of these formulae were generalised to the case of finite groups whose Sylow-$p$-subgroup is cyclic by Hughes and Kemper \cite{HughesKemper}, and to cyclic 2-groups in \cite{HimstedtSymonds}. 

Now let $H \leq G$ and let $M$ be a $\kk H$-module. Then $\kk G \otimes_{\kk H} M$ is naturally a $\kk G$-module, which we call the $\kk G$-module \emph{induced} from $M$, and write as $M \uparrow^G_H$. When $\alpha: M \rightarrow N$ is a $\kk H$-homomorphism we define $\alpha \uparrow^G_H: M\uparrow^G_H \rightarrow N\uparrow^G_H$ by $\alpha \uparrow^G_H := \id_{\kk G} \otimes_{\kk H} \alpha$.

A $\kk G$-module $M$ which is a direct summand of  $M \downarrow_H \uparrow^G_H$ is said to be projective relative to $H$, or simply projective if $H=1$. Other equivalent definitions will be given in Section \ref{sec:relstabmodcat}. More generally, for a set $\mathcal{X}$ of subgroups of $G$, a $\kk G$-module $M$ is said to be projective relative to $\mathcal{X}$ if it is a direct summand of $\oplus_{X \in \mathcal{X}} M \downarrow_X \uparrow_X^G$.

We will also show in section \ref{sec:relstabmodcat} that if $M$ is projective relative to $\mathcal{X}$ then $M^G = \sum_{X \in \mathcal{X}} \Tr^G_X(M^X)$. Consequently, elements of $I^G_{\mathcal{X}}$ are contained in summands of $\kk[V]$ which are projective relative to $\mathcal{X}$. 

The following results of Almkvist and Fossum concerning representations of cyclic groups of prime order are of particular interest to us.

\begin{Theorem}[Almkvist and Fossum]\label{thm:cyclicanalog}
Let $G=C_p$ be a cyclic group of prime order $p$ and let $\kk$ be a field of characteristic $p$. Let $V$ be the unique indecomposable representation of $G$ over $\kk$ with dimension 2 (with action given by a Jordan block of size two).

\begin{enumerate}
\item[(i)] (Projectivity)  Suppose $m,d<p$ and $m+d \geq p$. Then $S^d(S^m(V))$ is projective.

\item[(ii)] (Periodicity)  For any $m,d<p$ and any $r$ we have a $\kk G$-isomorphism $S^{pr+d}(S^m(V)) \cong S^d(S^m(V)) + \ \text{projective modules}$. 

%\item[(iii)] (Reciprocity)  We have a $\kk G$-isomorphism $S^d(S^m(V)) \cong S^m(S^d(V))$.
\end{enumerate}
\end{Theorem}

Of course, in determining the indecomposable summands of any modular representation of $C_p$, one is helped enormously by the fact that we have a classification of indecomposable representations. Indeed, the modules $V_{d+1}: = S^d(V)$ for $0 \leq d <p$, form a complete set of isomorphism classes of indecomposable modular representations for $C_p$, and $V_p$ is the unique projective indecomposable. Furthermore, each has a $C_p$-fixed subspace of dimension 1, and so the number of indecomposable summands in a given representation is equal to the dimension of the subspace fixed by $C_p$. For representations of elementary abelian $p$-groups, neither of these helpful results hold. In fact, if $G$ is an elementary abelian $p$-group of order $p^n$, then unless $n=1$ or $p=n=2$, the representation type of $G$ is ``wild''; essentially this means that there is no hope of classifying the indecomposable representations up to isomorphism. In spite of this, we prove the following:

\begin{Theorem}\label{thm:relprojinhighdegrees} Let $E$ be an elementary abelian $p$-group of order $q$, $\kk$ a field of characteristic $p$ and $V$ a faithful indecomposable $\kk E$-module with dimension 2. Let $m,d<q$ with $m+d \geq q$. Then $S^d(S^m(V)^*)$ is projective relative to the set of subgroups of $E$ with order $ \leq m$.
\end{Theorem}
Note that in case $E$ is cyclic of order $p$, $S^m(V)$ is self-dual, so this generalises the first part of Theorem \ref{thm:cyclicanalog}. We also prove the following generalisation of the second part:

\begin{Theorem}\label{thm:periodicmodulorelproj} Retain the notation of Theorem \ref{thm:relprojinhighdegrees}. Let $m,d$ be integers with  $m<q$, $d>q$ and let $d' \equiv d \mod q$ with $d'<q$. Then $S^d(S^m(V)^*) \cong S^{d'}(S^m(V)^*)$ modulo summands which are projective relative to the set of subgroups of $E$ with order $ \leq m$.
\end{Theorem}

\subsection{Structure of the paper}
In order to prove results like the two theorems above, we need to determine the isomorphism type of $\kk G$-modules up to the addition of other $\kk G$-modules which are projective relative to certain families of subgroups. Given any set $\mathcal{X}$ of subgroups of $G$, one can define a category whose objects can be viewed as residue classes of $\kk G$-modules up to the addition of relatively $\mathcal{X}$-projective modules. This is called the $\mathcal{X}$-relative stable module category. In case $\mathcal{X} = \{1\}$ it reduces to the stable module category, which has been written about extensively by many authors. We recommend \cite{CarlsonBook} as a good reference. Many familiar results about the stable module category generalise in a straightforward manner. In section 2 we define the $\mathcal{X}$-relative stable module category and gather together the results we need, in most cases omitting proofs. The goal of the section is to prove a result (Corollary \ref{cor:frob}) which says something about the relationship between stable module categories relative to different families of subgroups. 

The main body of work in this paper is section 3, in which we prove Theorems \ref{thm:relprojinhighdegrees} and \ref{thm:periodicmodulorelproj}. In section 4 we turn our focus to invariant theory, in particular proving Theorems \ref{thm:2q-3} and \ref{thm:quotientofkvg}.

\begin{ack} Special thanks go to Professor David Benson a number of invaluable conversations at the genesis of this work. Thanks also to Dr. M\"ufit Sezer for his assistance with the proof of Proposition \ref{prop:coinvbound}.
\end{ack}

\section{Relative projectivity and the relative stable module category}\label{sec:relstabmodcat}

In this section, we fix a prime $p>0$ and let $G$ be a finite group of order divisible by $p$. Let $\kk$ be a field of characteristic $p$ and let $\mathcal{X}$ be a set of subgroups of $G$.
Now let $M$ be a finitely generated $\kk G$-module. $M$ is said to be \emph{projective relative to $\mathcal{X}$} if the following holds: let $\phi:M \rightarrow Y$ be a $\kk G$-homomorphism and $j: X \rightarrow Y$ a surjective $\kk G$-homomorphism which splits on restriction to any subgroup of $H \in \mathcal{X}$. Then there exists a $\kk G$-homomorphism $\psi$ making the following diagram commute.

\begin{center}
\begin{tikzpicture}
\draw[->] (0,0) node[left] {$X$} -- (2,0) node[right] {$Y$}; 
\draw[->] (2.5,0) -- (4.5,0) node[right] {0};
\draw[->] (2.2,2) node[above] {$M$} -- (2.2,0.3);
\draw[dashed,->] (2,2) -- (0,0.3);
\draw (0.8,1.2) node[above]{$\psi$};
\draw (1,0) node[above] {$j$};
\draw (2.2,1) node[right] {$\phi$};
\end{tikzpicture}
\end{center}  

Dually, one says that $M$ is \emph{injective relative to $\mathcal{X}$} if the following holds: given an injective $\kk G$-homomorphism $i: X \rightarrow Y$ which splits on restriction to each $H \in \mathcal{X}$ and a $\kk G$-homomorphism $\phi:  X \rightarrow M$, there exists a $\kk G$-homomorphism $\psi$ making the following diagram commute.

\begin{center}
\begin{tikzpicture}
\draw[->] (0,0) node[left] {$X$} -- (2,0) node[right] {$Y$}; 
\draw[<-] (-0.5,0) -- (-2.5,0) node[left] {0};
\draw[->] (-0.2,-0.2) -- (-0.2,-2)node[below] {$M$};
\draw[dashed,->] (2,-0.2) -- (0.2,-2);
\draw (1.2,-1.2) node[below]{$\psi$};
\draw (1,0) node[above] {$i$};
\draw (-0.2,-1) node[left] {$\phi$};
\end{tikzpicture}
\end{center}  

These notions are equivalent to the usual definitions of projective and injective $\kk G$-modules when we take $\mathcal{X} = \{1\}$. We will say a $\kk G$-homomorphism is $\mathcal{X}$-split if it splits on restriction to each $H \in \mathcal{X}$.  Since a $\kk G$-module is projective relative to $H$ if and only if it is also projective relative to the set of all subgroups of $H$, we usually assume $\mathcal{X}$ is closed under taking subgroups.

Relative projectivity is associated very closely with the transfer maps defined in subsection \ref{subsec:transfer}. Given $\kk G$-modules and $M$ and $N$, there is a natural left action of $G$ on $\Hom_{\kk}(M,N)$ defined by $$(\sigma \cdot \alpha) v = \sigma\alpha (\sigma^{-1}v), \qquad \sigma \in G , \alpha \in \Hom_{\kk}(M,N), v \in M.$$
 We write $(M,N)$ for $\Hom_{\kk}(M,N)$, so that $(M,N)^G = \Hom_{\kk G}(M,N)$. If $H$ is a subgroup of $G$, then the map $\Tr^G_H: (M,N)^H \rightarrow (M,N)^G$ will be defined as
\[\Tr_H^G(\alpha) (v) = \sum_{\sigma \in S} \sigma \alpha (\sigma^{-1} v).\]
There is also a map $\res^G_H: (M,N)^G \rightarrow (M,N)^H$ obtained by restricting homomorphisms. We note the following properties of transfer:

\begin{Lemma}\label{transferprops}
\begin{enumerate}
\item  Let $\alpha \in (M,N)^H$ and $\beta \in (M,M)^G$. Then $\Tr_H^G(\alpha) \circ \beta = \Tr(\alpha \circ \res^G_H(\beta))$.
\item  Let $\alpha \in (N,N)^G$ and $\beta \in (M,N)^H$. Then $\alpha \circ \Tr^G_H(\beta) = \Tr(\res^G_H(\alpha) \circ \beta)$.
\end{enumerate}
\end{Lemma}

\begin{proof} See \cite[Lemma 3.6.3(i), (ii)]{Benson1}.
\end{proof}

There are many equivalent ways to characteristic relative projectivity:
\begin{Prop}\label{higman}
Let $G$ be a finite group of order divisible by $p$, $\mathcal{X}$ a set of subgroups of $G$ and $M$ a $\kk G$-module. Then the following are equivalent:
\begin{enumerate}
\item[(i)] $M$ is projective relative to $\mathcal{X}$;
\item[(ii)] Every $\mathcal{X}$-split epimorphism of $\kk G$-modules $\phi: N \rightarrow M$ splits;
\item[(iii)] $M$ is injective relative to $\mathcal{X}$;
\item[(iv)] Every $\mathcal{X}$-split monomorphism of $\kk G$-modules $\phi: M \rightarrow N$ splits;
\item[(v)] $M$ is a direct summand of $\oplus_{H \in \mathcal{X}} M \downarrow_H \uparrow^G$;
\item[(vi)] $M$ is a direct summand of a direct sum of modules induced from subgroups in $\mathcal{X}$
\item[(vii)] There exists a set of homomorphisms $\{\beta_{H}: H \in \mathcal{X}\}$  such that $\beta_H \in (M,M)^H$ and $\sum_{H \in \mathcal{X}} \Tr_H^G(\beta_H) = \Id_M$.
\end{enumerate}
\end{Prop}
 
The last of these is called \emph{Higman's criterion}.
\begin{proof} The proof when $\mathcal{X}$ consists of a single subgroup of $G$ can be found in \cite[Proposition~3.6.4]{Benson1}. This can easily be generalised. 
\end{proof}

Note that (vi) tells us that $M$ is projective relative to $\mathcal{X}$ if and only if $M$ decomposes as a direct sum of modules, each of which is projective to some single $H \in \mathcal{X}$.
The following corollary now follows immediately from \cite[Corollary~3.6.7]{Benson1}.
\begin{Corollary}\label{tensorrelproj} Suppose $M$ and $N$ are $\kk G$-modules and $N$ is projective relative to $\mathcal{X}$. Then $M \otimes N$ is projective relative to $\mathcal{X}$.
\end{Corollary}

Let $M$ and $N$ be $\kk G$-modules and let $\mathcal{X}$ be a set of subgroups of $G$.  Let $(M,N)^{G, \mathcal{X}}$ denote the linear subspace of $(M,N)^G$ consisting of homomorphisms which factor through some $\kk G$-module which is projective relative to $\mathcal{X}$. We consider the quotient
$$(M,N)^G_{\mathcal{X}} = (M,N)^G/(M,N)^{G,\mathcal{X}}.$$
One can define a category in which the objects are the $\kk G$-modules and $(M,N)^G_{\mathcal{X}}$ is the set of morphisms between $\kk G$-modules $M$ and $N$. This is called the $\mathcal{X}$-relative stable module category, or $_{\mathcal{X}}\stmod_{\kk G}$ for short. It reduces to the usual stable module category when we take $\mathcal{X} = \{1\}$.  

The question of whether a homomorphism factors through a relatively projective module is also related to the transfer.
\begin{Lemma}\label{factorsthrough} Let $M$, $N$ be $\kk G$-modules, $\mathcal{X}$ a collection of subgroups of $G$, and $\alpha \in (M,N)^G$. Then the following are equivalent:

\begin{enumerate}
\item $\alpha$ factors through $\oplus_{H \in \mathcal{X}} M \downarrow_H \uparrow^G$.
\item $\alpha$ factors through some module which is projective relative to $\mathcal{X}$.
\item There exist homomorphisms $\{\beta_H \in (M,N)^H: H \in \mathcal{X}\}$ such that $\alpha = \sum_{H \in \mathcal{X}} \Tr^G_H(\beta_H)$. 
\end{enumerate}
 \end{Lemma}

\begin{proof} This is easily deduced from \cite[Proposition 3.6.6]{Benson1}.
\end{proof}

If $\alpha, \beta \in (M,N)^G$, we will write $\alpha \equiv_{\mathcal{X}} \beta$ whenever $\alpha$ and $\beta$ are equivalent as morphisms in $_{\mathcal{X}}\stmod_{\kk G}$. In other words, whenever $\alpha - \beta \in \oplus_{X \in \mathcal{X}} \Tr_X^G(M,N)^X$. A homomorphism $\alpha \in (M,N)^G$ induces an isomorphism in $_{\mathcal{X}}\stmod_{\kk G}$ if and only if there exists a homomorphism $\beta \in (N,M)^G$ with the property that $\alpha \circ \beta \equiv_{\mathcal{X}} \id_N$ and $\beta \circ \alpha \equiv_{\mathcal{X}} \id_M$. We shall write  $M \simeq_{\mathcal{X}} N$ to say that $M$ and $N$ are isomorphic as objects in $_{\mathcal{X}}\stmod_{\kk G}$. The following is now easy to deduce:

\begin{Lemma} Let $M$ and $N$ be $\kk G$-modules. Then the following are equivalent:
\begin{enumerate}
\item $M \simeq_{\mathcal{X}} N$;
\item There exist $\kk G$-modules $P$ and $Q$ which are projective relative to $\mathcal{X}$ such that $M \oplus P \cong N \oplus Q$.
\end{enumerate}
\end{Lemma}

The next result, which we will need in the proof of our main theorem, is a generalisation of \cite[Lemma~3.1]{SymondsCyclic}.
\begin{Prop}\label{prop:frob} Let $G$ be a finite group, $\mathcal{X}$ a set of subgroups of $G$ which is closed under taking subgroups, and $\mathcal{Y}$ a non-empty subset of $\mathcal{X}$. Let $M$ and $N$ be $\kk G$-modules and suppose that either $M$ or $N$ is projective relative to $\mathcal{X}$. Suppose $\alpha \in (M,N)^G$ has the property that $\res^G_H(\alpha)$ factors through a module which is projective relative to the set $H \cap \mathcal{Y}:= \{K \in \mathcal{Y}: K \subseteq H\}$, for every $H \in \mathcal{X}$. Then $\alpha$ factors through a module which is projective relative to $\mathcal{Y}$. 
\end{Prop}

\begin{proof} We give the proof when $M$ is projective relative to $\mathcal{X}$; the proof when $N$ is projective relative to $\mathcal{X}$ is similar. By Lemma \ref{factorsthrough}, we can write, for each $H \in \mathcal{X}$
\[\res^G_H(\alpha) = \sum_{K \in H\cap \mathcal{Y}} \Tr^H_K(\beta_{H,K})\] where $\beta_{H,K} \in (M,N)^K$. Since $M$ is projective relative to $\mathcal{X}$ we can write
\[\Id_M = \sum_{H \in \mathcal{X}} \Tr^G_H(\mu_H)\] for some set of homomorphisms $\{\mu_H \in (M,M)^H: H \in \mathcal{X}\}$. Now we have
\[\alpha = \alpha \circ Id_M = \alpha \circ(\sum_{H \in \mathcal{X}} \Tr^G_H(\mu_H))\]
\[= \sum_{H \in \mathcal{X}} \Tr^G_H(\res^G_H(\alpha)\circ \mu_H)\]
by Lemma \ref{transferprops}(2),
\[= \sum_{H \in \mathcal{X}} \Tr^G_H(\sum_{K \in H \cap \mathcal{Y}} \Tr_K^H(\beta_{H,K}) \circ \mu_H)\]
\[= \sum_{H \in \mathcal{X}} \Tr^G_H(\sum_{K \in H \cap \mathcal{Y}} \Tr_K^H(\beta_{H,K} \circ \res^H_K(\mu_H))\]
\[= \sum_{H \in \mathcal{X}}(\sum_{K \in H \cap \mathcal{Y}} \Tr^G_K(\beta_{H,K} \circ \res^H_K(\mu_H)) \in (M,N)^{G,\mathcal{Y}}\]
as required.
\end{proof}

\begin{Corollary}\label{cor:frob} Let $\alpha, \beta \in (M,N)^G$. Suppose that, for all $H \in \mathcal{X}$, $\res^G_H(\alpha) \equiv_{H \cap \mathcal{Y}} \res^G_{H}(\beta)$. Then $\alpha \equiv_{\mathcal{Y}} \beta$.
\end{Corollary}

\begin{proof} Apply Proposition \ref{prop:frob} to $\alpha-\beta$.
\end{proof}

Given any $\kk G$-module $M$, it can be shown that there exists a surjective map $j: P \rightarrow M$, where $P$ is a relatively $\mathcal{X}$-projective $\kk G$-module and $j$ splits on restriction to each $H \in \mathcal{X}$. The minimal such $P$ is called the relatively $\mathcal{X}$-projective cover of $M$. This implies the existence, for any $\kk G$-module $M$, of a unique minimal resolution $P_* \rightarrow M$ by relatively $\mathcal{X}$-projective modules which splits on restriction to each $H \in \mathcal{X}$. One can then show, using an argument along the lines of \cite[Proposition~5.2]{CarlsonBook} that a pair of maps $\alpha, \beta \in (M,N)^G$ satisfy $\alpha \equiv_{\mathcal{X}} \beta$ if and only if the maps $\alpha_*$ and $\beta_*$ induced between minimal relatively $\mathcal{X}$-projective resolutions of $M$ and $N$ are chain homotopic. Dually, there exists an injective map $i: M \rightarrow Q$ where $Q$ is a relatively $\mathcal{X}$-projective $\kk G$-module and $i$ splits on restriction to each $H \in \mathcal{X}$. The minimal such $Q$ is called the relatively $\mathcal{X}$-injective hull of $M$, and leads to an analogous theory of minimal relatively $\mathcal{X}$-injective resolutions.

The categories $_\mathcal{X} \stmod_{\kk G}$ are not abelian categories; kernels and cokernels of morphisms are not well-defined. Rather, they are triangulated category. See \cite{NeemanTriangulated} for a full definition of a triangulated category. In a triangulated category there is a construction called a mapping cone, which replaces the cokernel. In $_\mathcal{X} \stmod_{\kk G}$ this works as follows: given any morphism $M \rightarrow N$ we choose a representative $\alpha \in (M,N)^G$. Now let $j: \ker(\alpha) \rightarrow M$ be the canonical inclusion and denote by $i:\ker(\alpha) \rightarrow Q$ the inclusion into the relatively $\mathcal{X}$-injective hull of $\ker(\alpha)$. Since $i$ splits on restriction to each $H \in \mathcal{X}$ there exists a $\theta \in (M,Q)^G$ such that $\theta \circ j = i$. Now define $\alpha': M \rightarrow N \oplus Q$ by $\alpha'(m) = (\alpha(m), \theta(m))$. One can check that $\alpha'$ is injective. Then the mapping cone of the original morphism is defined to be the cokernel of $\alpha'$. It can be shown in the fashion of \cite[Proposition~5.5]{CarlsonBook} (using the chain homotopy property of equivalent morphisms) that this construction is independent  of the choice of $\alpha$. Of course, if $\alpha$ is injective one can take $\alpha = \alpha'$. This implies

\begin{Lemma}\label{lemma:cokernels} Let $\alpha, \beta \in (M,N)^G$ with $\alpha \equiv_{\mathcal{X}} \beta$. Suppose $\alpha$ and $\beta$ are injective. Then $\coker(\alpha) \simeq_{\mathcal{X}} \coker(\beta)$.
\end{Lemma} 

We end this section with an elementary result which will be useful in section \ref{sec:invthy}.
\begin{Lemma}\label{relprojinduced} Let $M$ be a $\kk G$-module which is projective relative to a set $\mathcal{X}$ of subgroups of $G$. Then $M^G = \sum_{H \in \mathcal{X}} \Tr^G_H(M^H)$. 
\end{Lemma}

\begin{proof} It suffices to prove $M^G \subseteq \sum_{H \in \mathcal{X}} \Tr^G_H(M^H)$, the reverse inclusion being clear. As $M$ is projective relative to $\mathcal{X}$, there exists a set of homomorphisms $\{\beta_H \in (M,M)^H, H \in \mathcal{X}\}$ such that $\Id_M = \sum_{H \in \mathcal{X}} \Tr^G_H(\beta_H)$. Now let $v \in M^G$. As $v \in M^H$ for all $H \in \mathcal{X}$, we have $\beta_H(v) \in M^H$ for all $H \in \mathcal{X}$, and
 \[\Tr^G_H(\beta_H)(v) = \sum_{\sigma \in S} \sigma \beta_H (\sigma^{-1} v) = \sum_{\sigma \in S} \sigma \beta_H (v) = \Tr^G_H (\beta_H(v))  \]
 where $S$ is a left-transversal of $H$ in $G$. Therefore
 \[v = \Id_M(v) = \sum_{H \in \mathcal{X}} \Tr^G_H(\beta_H)(v) = \sum_{H \in \mathcal{X}} \Tr^G_H(\beta_H(v)) \in \sum_{H \in \mathcal{X}}\Tr^G_{H}(M^H) \]
 as required.
\end{proof}

%%%%%%%%%%%%%%%%%%%%%%%%%%%%%%%%%%%%%%%%%%%%
%%%%%%%%%%%%%%%%%%%%%%%%%%%%%%%%%%%%%%%%%%%

\section{Decomposing Symmetric Powers}

\subsection{Periodicity}\label{sec:periodicty}

We begin by describing a decomposition of symmetric powers applicable to all $p$-groups. 
Let $G$ be any finite $p$-group, $\kk$ a field of characteristic $p$, and let $V$ be any finite-dimensional indecomposable $\kk G$-module.   It is well-known that one may choose a basis $\{x_0,x_1,\ldots, x_m\}$ with respect to which the action of $G$ is lower-unitriangular, preserving the flag of subspaces $\langle x_m \rangle \subset \langle x_{m-1},x_m \rangle \subset \ldots \subset \langle x_0,x_1,\ldots,x_m \rangle $. We will refer to $x_0$ as the ``terminal'' variable and to $x_m$ as the ``initial'' variable. For any $x \in V$ we set $$N_G(x) = \prod_{y \in Gx} y$$ where $Gx$ denotes the orbit of $x$ under $G$. 

Now let $q$ denote the order of $G$, and let $B(V)$ be the set of all polynomials in $x_0,x_1,\ldots,x_m$ whose degree as a polynomial in $x_0$ alone is strictly less than $q$. $B(V) = \bigoplus_{d \geq 0}B^d(V)$ is graded by total degree. Since $G$ fixes the subspace $\langle x_1,x_2,\ldots, x_{m} \rangle$, $B(V)$ is a $\kk G$-submodule of $S(V)$. Further, given any $f \in S^d(V)$ with $x_0$-degree $\geq q$ we may perform  long division, writing uniquely $f = N_G(x_0)^a f' + b$ with $f' \in S(V)^{d-q}$ and $b \in B(V)$, where $a = q/|Gx_0|$.  We therefore obtain an isomorphism of graded $\kk G$-modules

%check the none-iterated versions of reference survive...
\begin{equation}\label{decomp} 
S(V)^d \cong N_G(x_0)^a \otimes S(V)^{d-q} \oplus B^d(V).
\end{equation}

\begin{rem}\label{propogate} Suppose $W$ is a direct summand of $S^d(V)$, and $f \in S^r(V)^G$. Then $f \otimes W$ is a submodule of $S^{d+r}(V)$ in general. One way of viewing the above is to say that $N_G(x_0) \otimes W$ is always a direct summand of $S^{d+q}(V)$. We sometimes say that $W$ is \emph{propagated} by the invariant $N_G(x_0)$. Note that if $W$ is projective, then since projective modules are injective we have that $f\otimes W$ is a direct summand of $S^{d+r}(V)$ for any $f \in S^r(V)^G$ - in other words, the projective direct summands are propagated by every invariant. 
\end{rem}

Now since multiplication with $x_m$ induces an injective map $S^d(V) \rightarrow S^{d+1}(V)$, there is an exact sequence of $\kk G$-modules
$$ 0 \longrightarrow S^{d}(V) \stackrel{\times x_m}{\longrightarrow} S^{d+1}(V) \longrightarrow S^{d+1}(V/x_m) \longrightarrow 0. $$
As multiplication by $x_m$ does not affect the $x_0$-degree and the second map does not increase it, this restricts  to an exact sequence of $\kk G$-modules
 $$ 0 \longrightarrow B^{d}(V) \stackrel{\alpha_d}{\longrightarrow} B^{d+1}(V) \longrightarrow B^{d+1}(V/x_m) \longrightarrow 0. $$

\subsection{Additive subgroups of fields of prime characteristic}\label{sec:additive}
Let $E$ be an elementary abelian $p$-group of order $q=p^n$ and $V$ a 2-dimensional faithful $\kk E$-module. As every representation of a $p$-group is conjugate to one in upper-unitriangular form, we may fix a basis $\{X,Y\}$ of $V$ such that the action of each $\alpha \in E$ is given by $\alpha \cdot X = X, \alpha \cdot Y = Y+\rho(\alpha) X$, where $\rho: E \rightarrow (\kk,+)$ is a homomorphism, which must be injective as $V$ is faithful. This allows us to regard $E$ as an additive subgroup of $\kk$ and prompts the study of such subgroups.

\begin{Lemma}\label{lemma:lidl} Let $G \leq \kk$ be an additive subgroup. Define the polynomial $$T_G(x) = \prod_{\alpha \in G}(x-\alpha).$$ Then $T_G(x)$ is a linearized polynomial, i.e. $$T_G(x) = \sum_{i=0}^n b_i x^{p^i}$$ for some coefficients $b_i \in \kk$.
\end{Lemma}
\begin{proof} See \cite[Theorem~3.52]{LidlNiederreiter}.
\end{proof}
The author thanks Jyrki Lahtonen for bringing this lemma to his attention.

\begin{Corollary}\label{cor:powersums} The power sum $$S_{i}(G)  = \sum_{\alpha \in G} \alpha^i$$ is zero for all $i<q-1$. Further, $S_{q-1}(G) \in \kk$ is not zero.
\end{Corollary}
\begin{proof} We have $T_G(x) = \sum_{j=0}^{q} e_{q-j} x^j$ where $e_j$ denotes the degree $j$ elementary symmetric polynomial in the elements of $G$. For $j<q$, Lemma \ref{lemma:lidl} implies that $e_j=0$ unless $j=q-p^m$ for some $ 0 \leq m < n$. Now the power sums may be expressed in terms of elementary symmetric polynomials by means of  the Newton-Girard identities; these are most readily written in matrix form as
\[S_{i}(G) = \begin{vmatrix} e_1 & 1 & 0 & \ldots & 0 \\ 2e_2 & e_1 & 1 & \ldots & 0 \\ \vdots & \vdots & \vdots & \vdots & \vdots \\ ie_i & e_{i-1} & e_{i-2} & \ldots & e_1   \end{vmatrix}.\]
Now we see straight away that the leftmost column consists entirely of zeroes if $i<q-1$, since for all $j \leq i$ we have either $e_j = 0$ or $j \equiv 0 \mod p$. We also see that $S_{q-1}(G) = -e_{q-1}$. Now $e_{q-1}$ is equal to $$\sum_{\beta \in G}\prod_{\alpha \in G, \alpha \neq \beta} \alpha.$$ But here the summands are all zero except when $\beta = 0$, hence $$e_{q-1} = \prod_{\alpha \in G, \alpha \neq 0} \neq 0.$$
\end{proof}

\subsection{Representations of elementary abelian $p$-groups}
For the rest of this section, let $E \leq \kk$ and let $V = \langle X, Y \rangle$ be a 2-dimensional faithful $\kk E$-module with action as in section \ref{sec:additive}. We want to study the modules $S^m(V)^*$. For any $i \leq m$ set $a_i = X^{m-i}Y^i$. Then the set $a_0,a_1,\ldots,a_m$ forms a basis of $S^m(V)$, and the action of $\alpha \in E$ on this basis is given by \begin{equation}\label{basisv} \alpha \cdot a_i = \sum_{j=0}^i {\tiny\begin{pmatrix} i \\ j  \end{pmatrix}} \alpha^ja_{i-j}. \end{equation} Notice that this does not depend on $m$; we have an inclusion $S^{m}(V) \subset S^{m+1}(V)$ for any $m \geq 0$.  Now let $x_0,x_1\ldots,x_m$ be the corresponding dual basis of $S^m(V)^*$; the action here is given by \begin{equation}\label{basisvstar} \alpha \cdot x_i = \sum_{j=0}^{m-i}{\tiny\begin{pmatrix} i+j \\ i  \end{pmatrix}} (-\alpha)^jx_{i+j}.\end{equation}

Note in particular that $x_m \in (S^m(V)^*)^E$ and $S^m(V)^*/\langle x_m \rangle \cong S^{m-1}(V)^*$. This follows because $x_m = S^{m-1}(V)^{\perp}$. We adopt the convention that for a natural number $r$, $rW$ denotes the direct sum of $r$ copies of $W$.   
\begin{Prop}\label{prop:smvstar} The following is true of the modules $S^m(V)^*$:
\begin{enumerate}

\item[(i)] $S^m(V)^*$ is indecomposable for $m \leq q-1$, and $S^{q-1}(V)^* \cong \kk E$.
\item[(ii)] $S^{qr+m}(V)^* \cong S^m(V)^* \oplus r \kk E$ for $m \leq q-1$ and any $r$.
\end{enumerate}
\end{Prop}

\begin{proof} By  \cite[Proposition~3.2]{CSWEltAbelian}, the ring of invariants $S(V)^E$ is a polynomial algebra generated by $X$ and $N_E(Y)$. Therefore the Hilbert Series of $S(V)^E$ is 
\begin{equation}\label{hilbertseries}
{ \frac{1}{(1-t)(1-t^q)}}.
\end{equation}
and so $\dim(S^m(V)^E)=1$ for $m \leq q-1$. Since a module is indecomposable if and only if its dual is, and since $\dim(S^m(V)^*)=m+1$ we obtain (i).\\
Now let $P$ denote the projective module $S^{q-1}(V)$, and let $B = \oplus_{i=0}^{q-2} S^{i}(V)$. We form the graded submodule $T = \bigoplus_{d \geq 0} T^d$ of $S(V)$ defined as
\begin{equation}\label{T} T = (\kk[X,N_E(Y)] \otimes P) \oplus (\kk[N_E(Y)] \otimes B), \end{equation} with grading induced from that on $S(V)$.
By Remark \ref{propogate} , $T$ is a direct summand of $S(V)$.  Clearly $T^{qr+m} \cong r(\kk E) \oplus S^m(V)$. The Hilbert series of $\kk[N_E(Y)]$ is $\frac{1}{1-t^q}$. As the dimension of $B$ in degree $k$ is $k+1$ if $k \leq q-2$ and zero otherwise, we have 
$$H(B,t) = 1+2t+3t^2+\ldots+(q-1)t^{q-2} = \frac{d}{dt}\left(\frac{1-t^q}{1-t}\right) = \frac{-qt^{q-1}}{1-t}+\frac{1-t^q}{(1-t)^2}.$$
Finally, as $P$ has dimension $q$ and lies in degree $q-1$, we have $H(P,t) = qt^{q-1}$.
Therefore $$H(T,t) = qt^{q-1}\frac{1}{(1-t)(1-t^q)} + \frac{1}{1-t^q} \left( \frac{-qt^{q-1}}{1-t}+\frac{1-t^q}{(1-t)^2} \right) = \frac{1}{(1-t)^2} = H(S(V),t).$$
Therefore $T = S(V)$. Taking duals on both sides gives the required result.
 \end{proof}

\iffalse
We need a little more information about the decomposition above. Suppose that $0<m<q$ and let $\{x_0,x_1,\ldots x_{qr+m}\}$ be a basis of $W = S^{qr+m}(V)^*$ such that the action of $E$ on $W$ is given by (\ref{basisvstar}). Use a lexicographic order on $W$ with $x_0 > x_1 > \ldots > x_{qr+m}$. We claim that

\fi

We need a little more information about the decomposition above. Suppose that $0 \leq m<q$ and let $\{x_0,x_1,\ldots x_{qr+m}\}$ be a basis of $W = S^{qr+m}(V)^*$ such that the action of $E$ on $W$ is given by (\ref{basisvstar}). Then we have
\[W/ \langle x_{qr+m} \rangle \cong S^{qr+m-1} \cong \left\{ \begin{array}{lr} rS^{q-1}(V)^* \oplus (S^{m-1}(V)^*) & m \neq 0; \\ rS^{q-1}(V)^*  & m = 0. \end{array} \right. \] This tells us immediately that $x_{qr+m}$ is contained in a summand of $S^{qr+m}(V)^*$ isomorphic to $S^m(V)^*$. Further, we observe that
\begin{Lemma}\label{lemma:wheresthefixedpoint}
The projective summand of $W$ is spanned by $$\{\alpha \cdot x_{iq}: \alpha \in E, i=0, \ldots r-1.\}$$
\end{Lemma} 

\begin{proof} Recall that for any $p$-group $P$ an indecomposable $\kk P$-module $M$ is projective if and only if $\Tr^P(M) \neq 0$. Further, $\Tr^E(\kk E)$ is one-dimensional. Now observe that for every $ 0 \leq i<r$ we have
\[\Tr^E (x_{iq}) = \sum_{j=0}^{m-iq} \begin{pmatrix} iq+j \\ iq \end{pmatrix} \left( \sum_{\alpha \in E} (-\alpha)^j \right) x_{iq+j}\]
\[ = \begin{pmatrix} iq+q-1 \\ iq \end{pmatrix} \left(\prod_{\alpha \in E, \alpha \neq 0} \alpha \right) x_{iq+q-1} + \text{lower degree terms} \] by Corollary \ref{cor:powersums}. The binomial coefficient here is equal to $1 \in \kk$ by Lucas' Theorem (see \cite{FineLucas}).
Consequently $\Tr(x_{iq}) \neq 0$ and $\{\alpha \cdot x_{iq}: \alpha \in E\}$ spans a projective indecomposable summand of $W$ for each $i$.
\end{proof}

Being projective, the modules $S^{q-1}(V)^*$ are permutation modules. It follows that their symmetric powers are also permutation modules. The next lemma helps identify the isomorphism classes of these permutation modules. For any $k \leq n$ we denote by $\mathcal{X}_k$ the set of subgroups of $E$ with order $\leq p^k$.
\begin{Lemma}\label{lemma:powersofproj}
Given $d>0$ we write $d=rp^k$ where $k \leq n$ is maximal such that $p^k$ divides $d$. Then we have
\begin{enumerate}
\item[(i)] If $k<n$, $S^d(\kk E) \simeq_{\mathcal{X}_k} 0$.
\item[(ii)] If $k=n$ then $S^d(\kk E) \simeq_{\mathcal{X}_{n-1}} \kk.$
\item[(iii)] For any $k$ we have more generally $$S^d(\kk E) \simeq_{\mathcal{X}_{k-1}} \bigoplus_{E' \leq E, |E'|=p^k} \frac{1}{p^{n-k}} \begin{pmatrix} p^{n-k}+r-1 \\ r \end{pmatrix} \kk \uparrow_{E'}^E .$$
\end{enumerate}
\end{Lemma}

%referee asks for a reference for "W is a permutation module"

\begin{proof}Let  W be a direct summand of $S^d(\kk E)$. Then $W$ is a permutation module; let $\{\sigma \cdot m: \sigma \in E\}$ be a basis of $W$, where $m$ is some monomial of degree $d$. Then $W$ has isomorphism type $\kk \uparrow_{E'}^E$ where $E'$ is the stabiliser of $m$.
 Clearly if the monomial $m$ has stabiliser $E'$ then $m$ can be written as a product of monomials of the form $\prod_{\sigma \in E'}(\sigma m')$. In particular, we must have that $|E'|$ divides $\deg(m)$.  This establishes (i). On the other hand, if $d=rp^n$ then there is a unique monomial with stabiliser $E$, namely $\prod_{\sigma \in E} x_{\sigma}^r$. This establishes (ii). 

Now let $E'$ be a  subgroup of $E$ with order $p^{k}$. Define a power series $P(E',t) = \sum_{d \geq 0} M^{E'}_d t^d$ where $M^{E'}_d$ is the number of monomials of degree $d$ fixed by $E'$. Then we have 
$$P(E',t) = \frac{1}{(1-t^{p^{k}})^{p^{n-k}}}$$
$$ = \sum_{r=0}^{\infty} \begin{pmatrix} p^{n-k}+r-1 \\ r \end{pmatrix} t^{rp^{k}}$$
by the generalised binomial theorem. Therefore the number of summands of $S^d(\kk E)$ with isomorphism type $\kk \uparrow_{E'}^E$ is $\frac{1}{p^{n-k}}\begin{pmatrix} p^{n-k}+r-1 \\ r \end{pmatrix}$, as each one spans a submodule of dimension $p^{n-k}$. As there are no trivial summands and all other summands are induced from smaller subgroups, we have proved (iii). 
\end{proof}

\subsection{Main results}
In order to make the main results more readable we introduce some more notation: we write $B_{d,m}$ for $B^d(S^m(V)^*)$, and $\alpha_{d,m}$ for the map $S^d(S^m(V)^*) \rightarrow S^{d+1}(S^m(V)^*)$ described in section \ref{sec:periodicty}.  By the remarks following equation \ref{basisvstar} we obtain, for any $d$ and $m<q$ an exact sequence
\begin{equation}\label{exactsequence}
 0 \longrightarrow B_{d,m} \stackrel{\alpha_{d,m}}{\longrightarrow} B_{d+1,m} \longrightarrow B_{d+1,m-1} \longrightarrow 0.
\end{equation}
When $E'<E$ is a proper subgroup we will write $B_{d,m}(E')$ for $B^d(S^m(V \downarrow_{E'})^*)$. Note that this is not the same thing as $B_{d,m} \downarrow_{E'}$; the the former consists of polynomials whose degree as a polynomial in the terminal variable of $S^m(V)^*$ is $<|E'|$ while the latter consists of polynomials whose degree as a polynomial in the terminal variable of $S^m(V)^*$ is $<q$. 

\begin{Prop}\label{special}
Let $d,m$ be a pair of positive integers with $m<q$ and $m+d \geq q$. Then the following hold:
\begin{enumerate}
\item[(i)] $B_{d,m}$ is projective relative to the set of proper subgroups of $E$. 
\item[(ii)] Assuming $n \geq 2$, let $s$ and $r$ be the quotients when $d$ and $m$ respectively are divided by $p^{n-1}$, with $d'$ and $m'$ the corresponding remainders. Then we have
\begin{equation}\label{propformula} B_{d,m} \simeq_{\mathcal{X}_{n-2}} \bigoplus_{E'<E: |E'|=p^{n-1}} \frac{1}{p}\left[ \begin{pmatrix} r+s \\ r \end{pmatrix} - \begin{pmatrix} r+s-p \\ r \end{pmatrix}\right] B_{d',m'}(E') \uparrow_{E'}^E \end{equation} provided $\frac{1}{p}\left[ \begin{pmatrix} r+s \\ r \end{pmatrix} - \begin{pmatrix} r+s-p \\ r \end{pmatrix}\right]$ is an integer, and $$B_{d,m}(E) \simeq_{\mathcal{X}_{n-2}} 0$$ otherwise.
\end{enumerate}
\end{Prop}

\begin{rem} The proof is by double induction. The first induction is on $n$, the rank of $E$. We will show that the (i) above holds when $n=1$. Then for the inductive step we will take a group $E$ of order $p^n$ and assume that both (i) and (ii) hold for all proper subgroups of $E$ - although in the $n=2$ case just (i) will be sufficient. We will then prove that  (ii) holds for $E$. Notice that this implies immediately that (i) holds for $E$.

For each fixed $n$, we will prove (ii) by backwards induction on $m$, starting at $m=q-1$. This means we will initially prove that (ii) holds for all pairs $(d,q-1)$. Then in the inductive step we will fix $m \leq q-1$ and assume that (ii) holds for all pairs $(d,m)$ such that $m+d \geq q$. We will prove (ii) holds for the pairs $(m-1,d+1)$. This part of the proof is the longest and relies on determining the equivalence class of the morphism $\alpha_{d,m}: B_{d,m} \rightarrow B_{d+1,m}$. 

In various parts the proof splits into two or more subcases, depending on the value of $d$ or $m$ modulo $q':= p^{n-1}$. This will be made clear in the text.  
\end{rem}

\begin{proof} {\bf Initial step, $n=1$.}\\ For the $n=1$ case, only the first statement needs to be checked. This states that $B_{d,m}$ is projective provided $m+d \geq p$ and $m<p$. When $n=1$, $E$ is a cyclic group and the proposition reduces to Theorem \ref{thm:cyclicanalog}; more precisely, to (i) when $d<p$ and to (ii) when $d \geq p$.

\noindent {\bf Inductive step for $n$.}\\
Now fix $n>1$, and assume that the proposition is true for all proper subgroups of $E$. The proof for each $n$ is by downward induction on $m$, starting at $q-1$. 

	{\bf Initial step: $m=q-1$.}\\
When $m=q-1$ we have $r=p-1$ and $m'= q'-1$. There are two cases to consider. 

{\bf Case 1: $d' \neq 0$, i.e. $d$ not divisible by $q'$.}\\
Since $m' =q'-1$, we have $m'+d' \geq q'$. Therefore, for every subgroup $E'\leq E$ with order $q'$, $B_{d',m'}(E')$ is projective relative to $\mathcal{X}_{n-2}$ by the inductive hypothesis on $n$. Now the proposition becomes simply `` $B_{d,q-1}$ is projective relative to $\mathcal{X}_{n-2}$''. We showed that, when $d$ is not divisible by $q'$, $S^d(\kk E)$ is projective relative to $\mathcal{X}_{n-2}$ in Lemma \ref{lemma:powersofproj}(iii). As $B_{d,q-1}$ is a direct summand of $S^d(S^{q-1}(V)^*)$ we get $B_{d,m} \simeq_{\mathcal{X}_{n-2}} 0$ as required.

{\bf Case 2: d' =0, i.e. $d$ is divisible by $q'$.}\\
In this case, $B_{d',m'}(E') = S^0(S^{q'-1}(V \downarrow_{E'})^*) = \kk$. Furthermore since $d \geq p^{n-1}$ we have $s \geq 1$, hence $r+s \geq p$. So by Lemma \ref{numbertheory} \[ \begin{pmatrix} r+s \\ r \end{pmatrix} \equiv \begin{pmatrix} r+s-p \\ r \end{pmatrix} \mod p\]
and we have to show that,
$$B_{d,q-1} \simeq_{\mathcal{X}_{n-2}} \bigoplus_{E'<E: |E|=p^{n-1}} \frac{1}{p}\left[ \begin{pmatrix} r+s \\ r \end{pmatrix} - \begin{pmatrix} r+s-p \\ r \end{pmatrix}\right] \kk \uparrow_{E'}^E.$$
Now by Lemma \ref{lemma:powersofproj}(iii), we have
\[S^d(S^{q-1}(V)^*) \simeq_{\mathcal{X}_{n-2}} \bigoplus_{E'<E: |E|=p^{n-1}} \frac{1}{p} \begin{pmatrix} r+s \\ r \end{pmatrix}  \kk \uparrow_{E'}^E\]
and $d-q=  (s-p)p^{n-1}$ so \[S^{d-q}(S^{q-1}(V)^*) \simeq_{\mathcal{X}_{n-2}}  \bigoplus_{E'<E: |E|=p^{n-1}}  \frac{1}{p} \begin{pmatrix} r+s-p \\ r \end{pmatrix} \kk \uparrow_{E'}^E\]
from which the result follows. This concludes the proof for $m=q-1$, and starts the induction on $m$.

%%%%%%%%%%%%%%%%%%%%%%%%%%%%%%%%%%%%%%%%%%%%%%%%%%%%%%%%%%%%%%%%%%%%%%%%

{\bf Inductive step for $m$:}\\
 Now fix $m \leq q-1$ and assume that the proposition is true for all pairs of the form $(d,m)$ such that and $m+d \geq q$. We must determine the equivalence class modulo $\mathcal{X}_{n-2}$ of $B_{d+1,m-1}$.

 We have $B_{d+1,m-1} = \coker(\alpha_{d,m}|_{B_{d,m}})$. We use $\alpha_{d,m}$ for this map when the context is clear. We want to determine $B_{d+1,m-1}$ up to the addition of $\kk E$-modules which are projective relative to $\mathcal{X}_{n-2}$; by Lemma \ref{lemma:cokernels} it is enough to compute $\coker(\alpha)$ where $\alpha \equiv_{\mathcal{X}_{n-2}} \alpha_{d,m}$.

Let the definitions of $d',m',s,r$ be as in the statement of the proposition. We have several cases to consider. 

{\bf Case 1: $d'+m' \geq q'$.}\\ In this case $B_{d',m'}(E')$ is projective relative to $\mathcal{X}_{n-2}$ by the inductive hypothesis on $n$, and by the inductive hypothesis on $m$, $B_{d,m}$ is too. This is the domain of $\alpha_{d,m}$, which is injective, so we must have $\alpha_{d,m} \equiv_{\mathcal{X}_{n-2}} 0$. It follows that 
\begin{equation}\label{equal} B_{d+1,m-1} \simeq_{\mathcal{X}_{n-2}} B_{d+1,m}.\end{equation}
  
This case now splits into subcases. Let us denote by $(d+1)'$ the remainder when $d+1$ is divided by $q'$.

{\bf Subcase 1a: $d' \neq q'-1$.}\\
 In this case, $(d+1)'=d'+1$, and since $m'+d'+1 \geq m'+d' \geq q'$ we get $B_{d'+1,m'}(E') \simeq_{\mathcal{X}_{n-2}} 0$ by the inductive hypothesis (i) on $n$. Now by the inductive hypothesis (ii) on $m$ we get that  $$B_{d+1,m} \simeq_{\mathcal{X}_{n-2}} 0.$$ So by (\ref{equal}) we get 
$$B_{d+1,m-1} \simeq_{\mathcal{X}_{n-2}} 0$$ too.

This is exactly what the proposition claims in this subcase. Note that $m'$ cannot be zero, since the assumption $m'+d' \geq q'$ then cannot hold. Therefore $(m-1)' = (m'-1)$, and since $(d+1)'+(m-1)' = d'+m' \geq q'$ we get $B_{d'+1,m'-1}(E') \simeq_{\mathcal{X}_{n-2}}$ by the inductive hypothesis on $n$. Then the proposition (ii) states that   $$B_{d+1,m-1}(E') \simeq_{\mathcal{X}_{n-2}}$$ which is what we have just shown. 

{\bf Subcase 1b: $d' = q'-1$.}\\
If $d' = q'-1$ then we have $(d+1)'=0$. Further, the quotient when $d+1$ is divided by $q'$ is then $s+1$. The assumption $m'+d' \geq q'$ rules out the possibility that $m' = 0$, so $(m-1)'=m'-1$ and the quotient when $m-1$ is divided by $q'$ is still $r$. So the difference of binomial coefficients appearing in the formula (\ref{propformula}) for $B_{d+1,m}$ is the same as the one appearing in the formula for $B_{d+1,m-1}$ and since $$B_{(d+1)',(m-1)'}(E') = B_{0,m'-1}(E') = \kk = B_{0,m'}(E') = B_{(d+1)',m'}$$ we get the desired equality. This ends the proof for case 1.

{\bf Case 2:  $m'+d<q'$.}\\
In this case, since $m+d \geq q$ we have
$$q \leq m+d = (r+s)p^{n-1}+m'+d' $$
$$\Rightarrow (r+s)p^{n-1} \geq q-m'-d' > p^n-p^{n-1} = p^{n-1}(p-1)$$
and therefore $r+s \geq p$. Since $m<q$ we get $r<p$ and by Lemma \ref{numbertheory}
\[\begin{pmatrix} r+s \\ r \end{pmatrix} \equiv \begin{pmatrix} r+s-p \\ r\end{pmatrix} \mod p.\]
Therefore by the inductive hypothesis on $m$, $$B_{d,m} \simeq_{\mathcal{X}_{n-2}} \bigoplus_{E'<E: |E'|=p^{n-1}} \frac{1}{p}\left[ \begin{pmatrix} r+s \\ r \end{pmatrix} - \begin{pmatrix} r+s-p \\ r \end{pmatrix}\right] B_{d',m'}(E') \uparrow_{E'}^E. $$ Note that this is the domain of $\alpha_{d,m}$.

{\bf Claim:} \begin{equation}\label{alpha}  \alpha_{d,m} \equiv_{\mathcal{X}_{n-2}} \alpha:= \bigoplus_{E'<E: |E'|=p^{n-1}} \frac1p \left[ \begin{pmatrix} r+s \\ r \end{pmatrix} - \begin{pmatrix} r+s-p \\ r \end{pmatrix}\right]\alpha_{d',m'}(E')\uparrow_{E'}^E.\end{equation}

{\it Proof of claim:} The inductive hypothesis on $m$ implies that $B_{d,m}$ is projective relative to $\mathcal{X}_{n-1}$. Applying Corollary \ref{cor:frob} with $\mathcal{X}= \mathcal{X}_{n-1}$ and $\mathcal{Y} = \mathcal{X}_{n-2}$ shows that it is enough to check that the formula (\ref{alpha}) is correct on restriction to each $E' < E$ with $|E'|=p^{n-1}$. Now the Mackey formula implies that for any $\kk E''$-module $W$ we have $$W \uparrow_{E''}^E \downarrow_{E'} \simeq_{\mathcal{X}_{n-2}} \left\{\begin{array}{lr} 0 & E' \neq E'' \\ pW & E'=E''.   \end{array}\right.$$ 
It follows that for any $E'<E$ with $|E'|=p^{n-1}$ we have $$\alpha \downarrow_{E'} \simeq_{\mathcal{X}_{n-2}}  \left[ \begin{pmatrix} r+s \\ r \end{pmatrix} - \begin{pmatrix} r+s-p \\ r \end{pmatrix}\right]\alpha_{d',m'}(E').$$

On the other hand, as $\kk E'$-modules $S^m(V)^* \cong S^{m'}(V)^* \oplus rS^{q'-1}(V)^*$ where $q' = p^{n-1}$. Following Proposition \ref{prop:smvstar}(ii), we take $\{x_0,x_{q'},\ldots,x_{(r-1)q'}\}$ as the $\kk E'$-module generators of the projective summand, and write $$S^m(V)^* = \bigoplus_{i=0}^{r-1} (E' \cdot x_{iq'}) \oplus M$$ where $M$ is the direct summand of $S^m(V)^*$ isomorphic to $S^{m'}(V)^*$.
The map $\alpha_{d,m} \downarrow_{E'}$ is induced by multiplication by an element $x$ of the fixed-point space of $M$, such that the quotient of $M$ by $x$ is isomorphic to $S^{m'-1}(V)^*$ if $m' \neq 0$ and the zero module otherwise. 

As $\kk E'$-modules we have
\[S^d(S^m(V)^*) = \bigoplus_{i_1+i_2+\ldots+i_r+j = q's+d'} S^{i_1}(E' \cdot x_0) \otimes \ldots \otimes S^{i_r}(E' \cdot x_{(r-1)q'}) \otimes S^j(M). \] 
 and 
\[B^d(S^m(V)^*) = \bigoplus_{i_1+i_2+\ldots+i_r+j = q's+d'} S^{i_1}(E' \cdot x_0)_{\{<q\}} \otimes \ldots \otimes S^{i_r}(E' \cdot x_{(r-1)q'}) \otimes S^j(M) \] 
where $S^{i_1}(E \cdot x_0)_{\{<q\}}$ means polynomials of degree $i_1$ in the linear expressions $\{\sigma \cdot x_0: \sigma \in E'\}$ whose degree as a polynomial in $x_0$ is $<q$.
%%%%%%%%%%%%%%%%%%%%%%%%%%%%%%%%%%%%%%%%%%%%%%%%%%%%%%%
Since $\alpha_{d,m} \downarrow_{E'}$ is injective, we can ignore any modules in the decomposition of its domain which are projective relative to $\mathcal{X}_{n-2}$. Note that $S^{*} (E' \cdot x_{kq'}) \simeq_{\mathcal{X}_{n-2}} \kk[N_{E'}(x_{kq'})]$ by the proof of Lemma \ref{lemma:powersofproj}(ii). As $N_{E'}(x_0)$ has $x_0$-degree $q'$ we have $$S^{*} (E' \cdot x_{0})_{\{<q\}} \simeq_{\mathcal{X}_{n-2}} \oplus_{i=0}^{p-1} \langle N_{E'}(x^i_{0}) \rangle.$$

 So the domain $B_{d,m}(E) \downarrow_{E'}$ is equivalent to a summand of 
\[\bigoplus_{j=0}^s \bigoplus_{\stackrel{i_1+i_2+\ldots+i_r=s-j}{i_1<p}} N_{E'}(x^{i_1}_0) \otimes \ldots \otimes N_{E'}(x^{i_r}_{(r-1)q'}) \otimes S^{jq'+d'}(M).    \]
 Further, $S^{jq'+d'}(M) \cong S^{jq'+d'}(S^{m'}(V \downarrow_{E'})^*) \simeq_{\mathcal{X}_{n-2}} S^{d'}(S^{m'}(V)^*) = B_{d',m'}(E')$ by the inductive hypothesis on $n$, for all values of $j$. Therefore forgetting the grading we have
\begin{equation}\label{bdm} B_{d,m}(E) \downarrow_{E'} \simeq_{\mathcal{X}_{n-2}} \bigoplus_{j=0}^s \bigoplus_{\stackrel{i_1+i_2+\ldots+i_r=s-j}{i_1<p}} \kk \otimes \kk \otimes \ldots \otimes \kk \otimes B_{d',m'}(E') \end{equation} and $\alpha_{d,m}\downarrow_{E'}$ on this is equivalent to 
\[\bigoplus_{j=0}^s \bigoplus_{\stackrel{i_1+i_2+\ldots i_r=s-j}{i_1<p}} \left(\id_{\kk} \otimes \id_{\kk} \ldots \otimes \id_{\kk} \otimes \ \alpha_{d',m'}(E') \right) \]

Obviously all the summands appearing in (\ref{bdm}) are isomorphic. The number of them is $$\sum_{j=0}^s \left( \ \text{Number of ways of writing $s-j$ as an ordered sum of $r$ non-negative integers, 
with the first $<p$. }   \right).$$
Let $\mathcal{P}(k,l)$ denote the number of ways of writing $k$ as an ordered sum of $l$ non-negative integers (where the order of summands is taken into account). An easy combinatorial argument shows that $\sum_{k=0}^l \mathcal{P}(k,l) = \begin{pmatrix} k+l \\ k \end{pmatrix}$. Evidently the number of summands appearing in  (\ref{bdm})  is $$\sum_{j=0}^s (\mathcal{P}(s-j,r) - \mathcal{P}(s-j-p,r)).$$ But clearly $\mathcal{P}(s-j-p,r) = 0$ if $j>s-p$, so the number of summands is \[ \sum_{j=0}^s \mathcal{P}(s-j,r) - \sum_{j=0}^{s-p}\mathcal{P}(s-j-p,r) = \begin{pmatrix} r+s \\r \end{pmatrix} - \begin{pmatrix} r+s-p \\ r \end{pmatrix}.\]

Now we have shown that 
\[\alpha_{d,m} \downarrow_{E'} \simeq_{\mathcal{X}_{n-2}} \left[ \begin{pmatrix} r+s \\ r \end{pmatrix} - \begin{pmatrix} r+s-p \\ r \end{pmatrix}\right]\alpha_{d',m'}(E') \equiv_{\mathcal{X}_{n-2}} \alpha\downarrow_{E'}\] and hence $\alpha_{d,m} \equiv_{\mathcal{X}_{n-2}} \alpha$ which proves our claim.\\

Now we must show that the cokernel of $\alpha_{d,m}$ is given by the formula  for $B_{d+1,m-1}$ in part (ii) of the proposition. Note that $\alpha_{d,m}$ has codomain $B_{d+1,m}$. There are three cases to check.

{\bf Subcase 2a: $m' \neq 0$:} Note that in this case it is not possible to have $d' = q'-1$, since we must have $m'+d'<q'$. Therefore the quotients when $m$ and $d+1$ are divided by $q'$ are $r$ and $s$ respectively, and by the inductive hypothesis on $m$ we have
$$B_{d+1,m} \simeq_{\mathcal{X}_{n-2}}  \bigoplus_{E'<E: |E'|=p^{n-1}} \frac{1}{p}\left[ \begin{pmatrix} r+s \\ r \end{pmatrix} - \begin{pmatrix} r+s-p \\ r \end{pmatrix}\right] B_{d'+1,m'}(E') \uparrow_{E'}^E $$ \begin{equation}\label{codom} = \bigoplus_{E'<E: |E'|=p^{n-1}} \frac{1}{p}\left[ \begin{pmatrix} r+s \\ r \end{pmatrix} - \begin{pmatrix} r+s-p \\ r \end{pmatrix}\right] \codom(\alpha_{d',m'}(E')) \uparrow_{E'}^E \end{equation}

and hence, using the formula (\ref{alpha}) for $\alpha$ we have $$B_{d+1,m-1} = \coker(\alpha_{d,m}) \simeq_{\mathcal{X}_{n-2}}  \bigoplus_{E'<E: |E'|=p^{n-1}} \frac{1}{p}\left[ \begin{pmatrix} r+s \\ r \end{pmatrix} - \begin{pmatrix} r+s-p \\ r \end{pmatrix}\right] \coker(\alpha_{d',m'}(E')) \uparrow_{E'}^E$$
$$=  \bigoplus_{E'<E: |E'|=p^{n-1}} \frac{1}{p}\left[ \begin{pmatrix} r+s \\ r \end{pmatrix} - \begin{pmatrix} r+s-p \\ r \end{pmatrix}\right] B_{d'+1,m'-1}(E') \uparrow_{E'}^E$$ as required.

{\bf Subcase 2b: $m' = 0, d' \neq q'-1$}
On the other hand, if $m'=0$ then $\alpha_{d',m'}: \kk \rightarrow \kk$ is an isomorphism in $_{\mathcal{X}_{n-2}}\stmod_{\kk E'}$ for every subgroup $E'<E$ with order $q'$. If in addition $d' \neq q'-1$ then as the quotient when $d+1$ is divided by $q'$ is still $s$, (\ref{codom}) still holds and $$\coker{\alpha_{d,m}} \simeq_{\mathcal{X}_{n-2}}  \bigoplus_{E'<E: |E'|=p^{n-1}} \frac{1}{p}\left[ \begin{pmatrix} r+s \\ r \end{pmatrix} - \begin{pmatrix} r+s-p \\ r \end{pmatrix}\right] \coker(\alpha_{d',m'}(E')) \uparrow_{E'}^E$$ $$ \simeq_{\mathcal{X}_{n-2}} 0.$$
This is in agreement with (\ref{propformula}) because $(m-1)'+(d+1)' =q'-1+d'+1 \geq q'$, hence by the inductive hypothesis on $n$ $B_{(d+1)',(m-1)'} \simeq_{\mathcal{X}_{n-2}} 0$. 

{\bf Subcase 2c: $m'=0, d'=q'-1$.}\\
 In this case, the quotient when $d+1$ is divided by $q'$ is not $s$ but $s+1$. We still have $r<p$ and $r+s+1 \geq p$, so by Lemma \ref{numbertheory}  $\left[ \begin{pmatrix} r+s+1 \\ r \end{pmatrix} - \begin{pmatrix} r+s+1-p \\ r \end{pmatrix}\right] $ is divisible by $p$. Therefore by the inductive hypothesis on $m$ we have $$B_{d+1,m} \simeq _{\mathcal{X}_{n-2}}  \bigoplus_{E'<E: |E'|=p^{n-1}} \frac{1}{p}\left[ \begin{pmatrix} r+s+1 \\ r \end{pmatrix} - \begin{pmatrix} r+s+1-p \\ r \end{pmatrix}\right] \kk \uparrow_{E'}^E.$$
The map $\alpha_{d',m'}(E')$ is again an isomorphism for each $E'<E$ with order $q'$, and so the image of $\alpha$ is contained in  $$\bigoplus_{E'<E: |E'|=p^{n-1}} \frac{1}{p}\left[ \begin{pmatrix} r+s \\ r \end{pmatrix} - \begin{pmatrix} r+s+p \\ r \end{pmatrix}\right] \kk \uparrow_{E'}^E.$$ Therefore the cokernel of $\alpha_{d,m}$ is $$ \simeq_{\mathcal{X}_{n-2}}\bigoplus_{E'<E: |E'|=p^{n-1}} \frac{1}{p}\left[ \begin{pmatrix} r+s \\ r-1 \end{pmatrix} - \begin{pmatrix} r+s+p \\ r-1 \end{pmatrix}\right] \kk \uparrow_{E'}^E.$$ This is what we want, because the quotient when $m-1$ is divided by $q'$ is $r-1$ and the quotient when $d+1$ is divided by $q'$ is $s+1$, so their sum is $r+s$, and $B_{(d+1)',(m-1)'} \simeq_{\mathcal{X}_{n-2}} B_{0,0} \cong \kk$.
\end{proof} 

In the above we used the following number-theoretic lemma:
\begin{Lemma}\label{numbertheory}
Let $r,s$ be integers and let $p$ be a prime. Suppose that $r<p$ and $r+s \geq p$. Then
\[\begin{pmatrix} r+s \\ r \end{pmatrix} \equiv \begin{pmatrix} r+s-p \\ r\end{pmatrix} \mod p\]
where the latter is interpreted as zero if $s<p$.
\end{Lemma}

\begin{proof}
If $s<p$ then since $s+r \geq p$ we have
\[\begin{pmatrix}
   r+s \\ r
  \end{pmatrix} = \frac{(r+s)(r+s-1)\cdots (p) \cdots (s+1)}{r(r-1)\cdots 2.1} \equiv 0 \mod p.
\]
While if $s \geq p$ we have\begin{align*} \begin{pmatrix}
   r+s-p \\ r
  \end{pmatrix} &= \frac{(r+s-p)(r+s-1-p)\cdots (s+1-p)}{r(r-1)\cdots 2\cdot 1} \\ &=  \frac{(r+s)(r+s-1) \cdots (s+1)}{r(r-1)\cdots 2 \cdot 1} \equiv \begin{pmatrix}
   r+s \\ r
  \end{pmatrix}  \mod p.
\end{align*}
\end{proof}
The following Corollary contains both Theorems \ref{thm:relprojinhighdegrees} and \ref{thm:periodicmodulorelproj} as special cases.

\begin{Corollary}\label{main}
Let $(d,m)$ be a pair of positive integers, with $m<p^k \leq q$ and $m+d \geq q$. Then $B_{d,m}$ is projective relative to $\mathcal{X}_{k-1}$. 
\end{Corollary}

\begin{proof} 
The proof is by backwards induction on $k$, the case $k=n$ having been covered in Proposition \ref{special}. Let $k \leq l \leq n$ and assume that $B_{d,m}$ is projective relative to $\mathcal{X}_{l}$ for all pairs $d,m$ with $d+m \geq q$ and $m<p^{l+1}$. Now suppose $m<p^l$ and $m+d \geq q$; we will show that $B_{d,m}$ is projective relative to $\mathcal{X}_{l-1}$. As $m<p^{l+1}$, we have that  $B_{d,m}$ is projective relative to $\mathcal{X}_{l}$. So by Proposition \ref{higman}(iv), $B_{d,m}$ is a direct summand of 
\begin{align*} & \bigoplus_{E' \in \mathcal{X}_l} (B_{d,m} \downarrow_{E'}) \uparrow^{E}_{E'} \\
 \simeq_{\mathcal{X}_{l-1}}& \bigoplus_{E'<E: |E'|=p^l} (B_{d,m}(E') \oplus B_{d-p^l,m}(E') \oplus \ldots \oplus B_{d-ap^l,m}(E')) \uparrow^{E}_{E'}
\end{align*}
where $a$ is the largest integer such that $a \leq (p^{n-l}-1)$ and $d-ap^l \geq 0$. Note that $$m+d-ap^l \geq m+d-p^l(p^{n-l}-1) = m+d-p^n+p^l \geq p^l,$$ therefore for each $E'$ the modules $B_{d,m}(E'),B_{d-p^l,m}(E'), \ldots, B_{d-ap^{l},m}(E')$ are projective relative to $\mathcal{X}_{l-1}$ by Proposition \ref{special} applied to $E'$, from which the result follows.
\end{proof}

\section{Applications to invariant theory}\label{sec:invthy}

Modular invariants of elementary abelian $p$-groups are a topic of much current interest in invariant theory - for example, \cite{CSWEltAbelian} describes generating sets for all such algebras of invariants for representations of dimension $2$, and in dimension 3 for groups of rank at most three. In this section we shall study the rings of invariants $\kk[V]^E$ where $E$ is an elementary abelian $p$-group of arbitary rank, $\kk$ is an infinite field of characteristic $p$ and $V \cong S^{m_1}(W) \oplus \ldots \oplus S^{m_r}(W)$ for some faithful indecomposable $\kk E$-module $W$ of dimension two, and for some set of integers $m_1,m_2, \ldots, m_r$ with $1 \leq m_i<q$ for all $i$. (We assume $m_i \geq 1$ for each, as clearly $\kk[V \oplus \kk]^G = \kk[V]^G \otimes \kk[x]$ where $x$ generates the trivial summand). Let $k$ be the smallest integer such that $m_i<p^k$ for all $i$. We view $E$ as an additive subgroup of $\kk$ as in section \ref{sec:additive}. Let  $x_{0,1}, x_{1,1}, \ldots x_{m_r,r}$ be the basis of $V^*$ such that the action of $\alpha \in E $ on $\{x_{0,i}, \ldots, x_{m_i,i}\}$ is given by the formula (\ref{basisvstar}) for all $i$, and let
\[N_i = N_{E}(x_{0,i}) = \prod_{\alpha \in E} \alpha \cdot(x_{0,i}).\] 
If $f \in \kk[V]$ then we shall say that $f$ is of multidegree $(d_1,d_2,\ldots, d_r)$ if $f$ has degree $d_i$ in  $\{x_{0,i}, \ldots, x_{m_i,i}\}$ for all $i$.
We have a decomposition
$$\kk[V]_{d_1,d_2,\ldots,d_r} \cong \kk[S^{m_1}(W)]_{d_1} \otimes \kk[S^{m_2}(W)]_{d_2} \otimes \ldots \otimes \kk[S^{m_r}(W)]_{d_r}.$$

Further, for each $i$ where $d_i \geq q$ we have $\kk[S^{m_i}(W)]_{d_i} \cong S^{d_i} (S^{m_i}(W)^*) \cong N^{s_i}_i \otimes S^{d'_i}(S^{m_i}(W)^*) \oplus B_{d_i,m_i}$
where $d'_i$ and $s_i$ are the remainder and quotient when $d_i$ is divided by $q$ and $B_{d_i,m_i}$ is the set of polynomials in $S^{d_i}(S^{m_i}(W)^*)$ whose degree in $x_{0,i}$ is $<q$. Notice that, by Corollary \ref{main}, $B_{d_i,m_i}$ is projective relative to $\mathcal{X}_{k-1}$, if $d_i \geq q-m_i$.
\begin{Prop}\label{prop:genset}
$\kk[V]^E$ has a generating set consisting of

\begin{enumerate}
 \item[(i)] The orbit products $N_i$, $i = 1,\ldots, r$;
 \item[(ii)] Certain invariants of multidegree $(d_1,d_2,\ldots, d_r)$, where $d_i<q-m_i$ for all $i$.
 \item[(iii)] Certain invariants of the form $\Tr^E_{H}(f)$ for $f \in \kk[V]^H$, where $H \in \mathcal{X}_{k-1}$.
\end{enumerate}
\end{Prop}

\begin{proof}
Let $f \in \kk[V]^E_{d_1,d_2,d_3,\ldots,d_r}$. If $d_i<q-m_i$ for all $i$ there is nothing to prove. If for some $i$ we have $q-m_i \leq d_i < q$ then
$$ \kk[V]^E_{d_1,d_2,d_3,\ldots,d_r}  \cong \kk[S^{m_1}(W)]_{d_1} \otimes \kk[S^{m_2}(W)]_{d_2} \otimes \ldots \otimes \kk[S^{m_i}(W)]_{d_i} \otimes \ldots \otimes \kk[S^{m_r}(W)]_{d_r}$$ is projective relative to $\mathcal{X}_{k-1}$ by the above discussion and Corollary \ref{tensorrelproj}. Then by Lemma \ref{relprojinduced} we have $f \in I^E_{\mathcal{X}_{k-1}}$. This completes the proof in case $d_i<q$ for all $i$. So now assume that $d_i \geq q$. The proof is now by induction on the total degree of $f$ (the case of total degree $<q$ being settled already). We can write
\[f = N_i^{s_i}f' + b\]
for some unique $f' \in \kk[V]_{d_1,d_2,\ldots,d'_i,\ldots, d_r}$ and $b \in  \kk[V]_{d_1,d_2,\ldots,d_i,\ldots, d_r}$ whose degree in $x_{0,i}$ is $<q$. Furthermore for any $\alpha \in E$ we have
\[f = \alpha \cdot f = N_i^{s_i}(\alpha \cdot f') + \alpha \cdot b\]
so the uniqueness of division with remainder implies that $f' $ and $b$ are invariant. By induction, $f'$ belongs to the subalgebra of $\kk[V]^E$ generated by the claimed generating set and $b$ is a fixed point in \[\kk[S^{m_1}(W)]_{d_1} \otimes \ldots \otimes \kk[S^{m_{i-1}}(W)]_{d_{i-1}} \otimes B_{d_i,m_i} \otimes \kk[S^{m_{i+1}}(W)]_{d_{i+1}} \otimes \ldots \otimes \kk[S^{m_{r}}(W)]_{d_{r}}\]
which, by Corollary \ref{tensorrelproj} and  Theorem \ref{main} is projective relative to $\mathcal{X}_{k-1}$. Then by Lemma \ref{relprojinduced} we have $b \in I^E_{\mathcal{X}_{k-1}}$. This completes the proof.
\end{proof}

In the case $q=p$ the above result is due to Wehlau \cite{WehlauCyclicViaClassical}, who also obtains more information about the invariants of type (ii) appearing in the generating set. Note that, since for a cyclic group $E$ of order $p$ every $\kk E$-module can be decomposed into one of the form $S^{m_1}(W) \oplus \ldots \oplus S^{m_r}(W)$ with $W$ the unique indecomposable of dimension 2, Wehlau's result applies to all modular representations of cyclic groups of prime order. Contrastingly, we do not know whether Proposition  \ref{prop:genset} can be generalised to arbitrary modular representations of elementary abelian $p$-groups.

In the case $r=1$ we obtain
\begin{Prop}\label{prop:generalquotient} Let $m<p^k\leq q$ and let $\mathcal{X}$ be the set of subgroups of $E$ with order $< p^k$. Let $V = S^m(W)$. Then the quotient algebra $\kk[V]^E/I^E_{\mathcal{X}}$ is generated by images of invariants of degree at most $q$.
\end{Prop}
Taking $k=n$ above in particular implies Theorem \ref{thm:quotientofkvg}.

\subsection{Coinvariants and degree bounds}
Let $G$ be a finite group of order divisible by $p$ and $V$ a finite-dimensional $\kk$-vector space. The Hilbert Ideal $\mathcal{H}$ of $\kk[V]$ is defined to be the ideal generated by positive degree invariants, i.e. $\kk[V]^G_+ \kk[V]$. The algebra of coinvariants $\kk[V]_G$ is defined to be the quotient $\kk[V]^G/\mathcal{H}$, or equivalently as $\kk[V] \otimes_{\kk[V]^G} \kk$. This is a finite-dimensional $\kk G$-module.

Since the map $\Tr^G$ is $\kk[V]^G$-linear, it follows that $\Tr^G$ maps a vector space basis for $\kk[V]_G$ to a generating set of the ideal $I^G_{1}$. This observation was used to compute the Noether numbers for arbitary modular representations of cyclic groups of order $p$ in \cite[Corollary~3.4]{FleischmannSezerShankWoodcock}. We want a similar result for elementary abelian $p$-groups. 

We use the notation of the previous subsection, so let $E$ be an elementary abelian $p$-group of order $q=p^n$, and $W$ a faithful indecomposable $\kk E$-module of dimension 2. Let $V = S^{m_1}(W) \oplus \ldots \oplus S^{m_r}(W)$, where $1 \leq m_i<p$ for all $i=1, \ldots, r$.  Recall that we may identify $E$ with a subgroup of $\kk$ and choose a basis $x_{0,1}, x_{1,1}, \ldots x_{m_r,r}$  of $V^*$ such that the action of $\alpha \in E$ is given by the formula ($\ref{basisvstar})$.  Recall that $\{x_{1,1},x_{2,1} , \ldots, x_{m_1,1}, \ldots, x_{m_r,r}\}$ is a $\kk E$-submodule of $V^*$, and let $A$ be the $\kk G$-subalgebra of $S(V^*)$ generated by these variables. We use a graded lexicographic order on $S(V^*)$ with $x_{m_i,i}<x_{m_i-1,i}< \ldots < x_{0,i}$ for all $i$.

\begin{Prop}\label{prop:coinvbound}
Let $m$ be a monomial of degree $q-1$ in $A$. Then $m$ is the lead term of an element of $\mathcal{H}$.
\end{Prop}

\begin{proof} Write $m= \prod_{j=1}^{q-1} u_j$ where for each $j$ we have $u_j=x_{i(j),t(j)}$ for some $t(j) =1, \ldots, r$ and $i(j)=1 , \ldots, m_{t(j)}$. For each $j$ we define $u'_j = x_{i(j)-1,t(j)}$, and write $m' = \prod_{j=1}^{q-1}u'_j$. Now for each $S \subseteq \{1,2,\ldots, q-1\}$ we define $S':= \{1,2,3,\ldots, q-1\} \backslash S$ and $X_S:= \prod_{j \in S} u'_j$. Now define
$$F = \sum_{\alpha \in E} \prod_{j=1}^{q-1} (u'_j - \alpha \cdot u'_j).$$
On the one hand, we have $$\prod_{j=1}^{q-1} (u'_j - \alpha \cdot u'_j) = \prod_{S \subseteq \{1,2, \ldots, q-1\}} (-1)^{|S|}X_S (\alpha \cdot X_{S'}).$$ Therefore $$F = \prod_{S \subseteq \{1,2, \ldots, q-1\}} (-1)^{|S|}X_S \Tr^E( X_{S'})$$ which shows that $F \in \mathcal{H}$. 
Note that $$\alpha \cdot u'_j = \alpha \cdot x_{i(j)-1,t(j)} = x_{i(j)-1,t(j)}  - \alpha i(j) x_{i(j),t(j)} + \text{terms of lower degree}.$$ Since $m_t<p$ for all $t=1,\ldots, r$, the integer $i(j)$ is not zero in $\kk$. It follows that the lead term of $ u'_j - \alpha \cdot u'_j$ is $- \alpha i(j) u_j$. Therefore the lead term of $F$ is $\sum_{\alpha \in E} (-\alpha)^{q-1} \lambda m$, where $\lambda = \prod_{j=1}^{q-1} i(j)$ is a non-zero element of $\mathbb{F}_p \subset \kk$. By Corollary \ref{cor:powersums}, $\sum_{\alpha \in E} (-\alpha)^{q-1} = \mu$ is a nonzero element of $\kk$ and hence the lead term of $F$ is $\mu \lambda m$. Dividing $F$ by $\mu \lambda$ then produces an element of $\mathcal{H}$ with lead term $m$. 
\end{proof}

 \begin{Corollary} The top degree of the coinvariants $\kk[V]_E$ is bounded above by $q-2+r(q-1)$. 
\end{Corollary}
\begin{proof} It is well known that, with respect to any graded ordering of variables, the Hilbert series of a graded ideal and its ideal of lead terms coincide. Therefore it suffices to prove that any monomial of degree $>q-2+r(q-1)$ must be the lead term of an element of the Hilbert ideal. Let $m \in \kk[V]$ be a monomial which is not the lead term of an element of $\mathcal{H}$. Write $m=\prod_{i=1}^r x_{0,i}^{k_i}h$ where $h \in A$. By Proposition \ref{prop:coinvbound}, $\deg(h) \leq q-2$. Further, since $x_{0,i}^q$ is the lead term of $N_E(x_{0,i})$ we must have $k_i \leq q-1$ for each $i = 1, \ldots, r$. This completes the proof. 
\end{proof}

The case $r=1$ of the following Corollary is Theorem \ref{thm:2q-3}.

\begin{Corollary} The Noether number $\beta(\kk[V]^E)$ is  bounded above by $q-2+r(q-1)$.
\end{Corollary}

\begin{proof} Proposition \ref{prop:genset} with $k=1$ implies that $\kk[V]^E$ is generated by the orbit products $N_1,N_2, \ldots, N_r$ which have degree $q$, certain invariants of (total) degree $\leq rq - \sum_{i=1}^r (m_i+1) \leq r(q-2)$, and elements of $I^E_{1}$. Now the previous Corollary implies $\kk[V]_E$ has a vector space basis $f_0,f_1,\ldots, f_l$ consisting of polynomials of degree $\leq q-2+r(q-1)$. Therefore $I^E_{1}$ is generated as a $\kk[V]^E$ module by $\Tr^E(f_0), \Tr^E(f_1), \ldots , \Tr^E(f_l)$. The result now follows by induction on degree.
\end{proof}

Note that we need the condition $m_i<p$ for all $i$, otherwise the generating set provided by Proposition \ref{prop:genset} may contain elements of the form $\Tr^E_H(f)$ for nontrivial subgroups $H$ of $E$. It is fairly straightforward to show that the degrees of these transfers are bounded above by the top degree of the \emph{relative coinvariants} $_H\kk[V]_E:= \kk[V] \otimes_{\kk[V]^E} \kk[V]^H$, but we do not know a method of obtaining an upper bound for this at present.

\bibliographystyle{plain}
\bibliography{MyBib}

\end{document}